\documentclass[10pt]{amsart}
\usepackage{amssymb}
\usepackage{epsfig}
\usepackage{url}
\usepackage{setspace}
\theoremstyle{plain}

\newtheorem{thm}{Theorem}[section]
\newtheorem{cor}[thm]{Corollary}
\newtheorem{prob}[thm]{Problem}
\newtheorem{lem}[thm]{Lemma}

\newtheorem{rem}[thm]{Remark}

\newtheorem{exam}[thm]{Example}
\def\cal{\mathcal}
\def\bbb{\mathbb}
\def\op{\operatorname}
\renewcommand{\phi}{\varphi}

\newcommand{\N}{\bbb{N}}
\newcommand{\Z}{\bbb{Z}}
\newcommand{\Q}{\bbb{Q}}

\newcommand{\C}{\bbb{C}}

\makeatletter
\let\@@pmod\pmod
\DeclareRobustCommand{\pmod}{\@ifstar\@pmods\@@pmod}
\def\@pmods#1{\mkern4mu({\operator@font mod}\mkern 6mu#1)}
\makeatother

\begin{document}

\title[Diagonal quintic threefolds with infinitely many rational points]{Construction of diagonal quintic threefolds with infinitely many rational points}
\author{Maciej Ulas}


\begin{abstract} In this note we present a construction of an infinite family of diagonal quintic threefolds defined over $\Q$ each containing infinitely many rational points. As an application, we prove that there are infinitely many quadruples $B=(B_{0}, B_{1}, B_{2}, B_{3})$ of co-prime integers such that for a suitable chosen integer $b$ (depending on $B$), the equation $B_{0}X_{0}^5+B_{1}X_{1}^5+B_{2}X_{2}^5+B_{3}X_{3}^{5}=b$ has infinitely many positive integer solutions.
\end{abstract}

\maketitle

\section{Introduction}\label{sec1}
Let $A=(A_{0}, A_{1}, A_{2}, A_{3}, A_{4})\in \Z^{5}$ be given and assume that $\gcd(A)=1$. We consider the diagonal quintic form in five variables given by
$$
F_{A}(X_{0}, X_{1}, X_{2}, X_{3}, X_{4})=\sum_{i=0}^{4}A_{i}X_{i}^5,
$$
and a related diagonal quintic threefold defined as
$$
\cal{V}_{A}:\;F_{A}(X_{0}, X_{1}, X_{2}, X_{3}, X_{4})=0.
$$
The set of integer points on the variety $\cal{V}_{A}$ is denoted by $\cal{V}_{A}(\Z)$. We say that a point $(X_{0}, X_{1}, X_{2}, X_{3}, X_{4})\in\cal{V}_{A}(\Z)$ is non-trivial if the greatest common divisor of the entries $X_{0}, X_{1}, X_{2}, X_{3}, X_{4}$ is equal to 1, and no proper sub-sum in the expression $\sum_{i=0}^{4}A_{i}X_{i}^5$ vanish. In the sequel, by an integer point on $\cal{V}_{A}$ we will mean a non-trivial one.

The question concerning the existence of integer points, or equivalently rational points, on $\cal{V}_{A}$ is very difficult. In particular, if $I=(1,1,1,1,1)$, the variety $\cal{V}_{I}$ is called the Fermat quintic and it is an open question whether $\cal{V}_{A}(\Z)$ is infinite. Euler conjectured that there are no non-trivial points in $\cal{V}_{I}(\Z)$. However, in 1966, Lander and Parkin in \cite{LP} found that $(27, 84, 110, 133, -144)\in\cal{V}_{A}$, that is the following equality holds
$$
27^5 + 84^5 + 110^5 + 133^5 = 144^5.
$$
The computer search was extended by others and two additional points were find on $\cal{V}_{I}$. In 1997 Scher and Seidl found the point $(5027, 6237, 14068, -220, -14132)$. The other one $(55, 3183, 28969, 85282, -85359)$ was found by Frey. In fact, according to our best knowledge, there is no known example of a vector $A\in\Z^{5}$ such that $\cal{V}_{A}$ contains infinitely many non-trivial integer points and this motivated us to work on this problem. Note that $\cal{V}_{A}\simeq \cal{V}_{I}$ over $\C$. The geometry of $\cal{V}_{I}$ was investigated in many papers. In particular, the geometry of lines on $\cal{V}_{I}$ was investigated by Albano and Katz \cite{AK}. The question concerning the existence of (complex) conics on $\cal{V}_{I}$ was the subject of PhD dissertation of Xu \cite{X}, some experimental work of Testa  \cite{HBT} and recent work of A. M. Musta\c{t}\v{a} \cite{M}. In this context let us recall that a generic quintic threefold contains 2875 lines and 609,250 conics. However, the findings of the mentioned authors do not shed light on the question concerning non-emptiness of $\cal{V}_{A}(\Z)$.

Let us describe the content of the paper in some details. In Section \ref{sec2} we quickly eliminate the possibility of the existence of $A\in\Z^{5}$ with nonzero entries, such that the variety $\cal{V}_{A}$ contains a line defined over $\Q$ giving infinitely many nontrivial rational points. In Section \ref{sec3} we present construction of a quadratic parametric solution of the equation defining the variety $\cal{V}_{A}$ for suitable chosen values of $A$. As an immediate application we get lower bound for the number of those $A\in\Z^{5}$ satisfying $\gcd(A)=1$ and such that $\cal{V}_{A}$ has infinitely many integer points. More precisely, the number $\cal{P}(N)$, of such $A\in\Z^{5}$ with $\op{max}(|A|)\leq N$ and $\gcd(A)=1$, satisfies $\cal{P}(N)\gg N^{1/10}$. Finally, in the last section we prove that there are infinitely many quadruples $B=(B_{0}, B_{1}, B_{2}, B_{3})$ of co-prime integers such that for a suitable chosen integer $b$ (depending on $B$), the equation $B_{0}X_{0}^5+B_{1}X_{1}^5+B_{2}X_{2}^5+B_{3}X_{3}^{5}=b$ has infinitely many positive integer solutions.

\bigskip

In our work we extensively used symbolic calculations, especially Gr\"{o}bner bases computations. All Gr\"{o}bner bases were computed with the help of Mathematica 13.2 computational package \cite{Wol} running on a standard machine with i7 type processor and 32 GB of RAM.

\section{Warm up}\label{sec2}

The first thing which comes to mind is to find linear polynomials and an appropriate $A\in\Z^{5}$ such that on $\cal{V}_{A}$ there is a line defined over $\Q$. In other words, for variables $u, v$, we take $X_{i}=a_{i}u+b_{i}v$ for $i=0,\ldots, 4$ and look for rational numbers $a_{i}, b_{i}$ such that
$$
\sum_{i=0}^{4}A_{i}(a_{i}u+b_{i}v)^{5}=0
$$
in the ring $\Q[u, v]$. In other words, we have a line lying on $\cal{V}_{A}$. Because we are interested in the existence of nontrivial points, there are $i, j\in\{0, 1, 2, 3, 4\}$ such that $i\neq j$ and linear forms $a_{i}u+b_{i}v, a_{j}u+b_{j}v$ are not proportional. Thus, after linear change of variables, without loss of generality, we can assume that $i=0, j=1$, and $a_{0}=1, b_{0}=0, a_{1}=0, b_{1}=1$. Let $C_{i}$ be the coefficient of $u^{i}v^{5-i}$ in the expansion of $\sum_{i=0}^{4}A_{i}(a_{i}u+b_{i}v)^{5}$ and consider the ideal $J$ generated by the polynomials $C_{i}, i=0, 1, \ldots, 5$. Moreover, let $\op{Gb}(J)$ be a Gr\"{o}bner basis of the ideal $J$. Then, we have that
$$
\op{Gb}(J)\cap \Q[a_4,b_4,A_0,A_1,A_2,A_3,A_4]=\{a_4^2b_4^3 A_0 A_1 A_4,a_4^3b_4^2 A_0 A_1 A_4\},
$$
and quick computation reveals that any solution of the system $C_{i}=0, i=0, \ldots, 5$, leads to lines which generate only trivial points on $\cal{V}_{A}$. Let us note that essentially the same result can be obtained by using \cite[Proposition 1.3]{AK}, which says that if $L$ is a line contained in $\cal{V}_{A}, A=(1, 1, 1, 1, 1)$, then $L$
goes through a point with 3 coordinates equal to 0.

In any way, the failure of the above approach suggests that if we want to construct parametric solution of the equation defining $\cal{V}_{A}$, or equivalently, a rational curve lying on $\cal{V}_{A}$, we need to look for parametric solutions given by polynomials of degrees $>1$.

\section{The construction of a quadratic parametrization}\label{sec3}

In this section we will find one parameter family of 5-tuples $A$ such that $\cal{V}_{A}$ contains the curve defined by quadratic forms. It is tempting to work with the general form of such a curve, i.e., try to find values $A\in\Z_{5}$ and $a_{i}, b_{i}, c_{i}\in\Q$ such that $X_{i}(u, v)=a_{i}u^2+b_{i}uv+c_{i}v^2, i=0, \ldots, 4$, satisfy the identity
$$
0=\sum_{i=0}^{4}A_{i}(a_{i}u^2+b_{i}uv+c_{i}v^2)^{5}=\sum_{j=0}^{10}D_{j}u^{10-j}v^{j},
$$
in the ring $\Q[u, v]$. Treating $a_{i}, b_{i}, c_{i}, A_{i}$ for $i=0, \ldots, 4,$ as independent variables, we have that $D_{j}\in\Z[A_{0}, \ldots, A_{4}, a_{0}, \ldots, c_{4}]$. Thus, we are interested in finding rational solutions of the system
$$
\cal{S}:\;D_{j}=0, \quad j=0, 1, \ldots, 10.
$$
To be more specific note that we are playing with a system involving 20 variables and 11 homogeneous polynomials. A naive count $20-11-1=8$ suggests that the dimension of the variety defined by the system $\cal{S}$ is 8 and the expectation that $\cal{S}$ has solutions in rational numbers is reasonable. However, the resulting system is too complicated to be fully investigated. Although we were trying quite hard, we were unable to solve $\cal{S}$ without any additional assumptions on the shape of the form $X_{i}(u, v)$. Thus, instead of working with the general quadratic forms, we decided to work with forms
\begin{align*}
&X_{0}=s_{0}(u, v)=uv,\\
&X_{1}=s_{1}(u, v)=u^2-auv+v^2,\\
&X_{2}=s_{2}(u, v)=u^2+auv+v^2,\\
&X_{3}=s_{3}(u, v)=u^2+buv+cv^2,\\
&X_{4}=s_{4}(u, v)=u^2+duv+ev^2,
\end{align*}
where $a, b, c, d, e\in\Q$ need to be determined. Thus, let $C_{i}$ be the $i$-th coefficient in the expansion
$$
\sum_{i=0}^{4}A_{i}s_{i}(u, v)^5=\sum_{i=0}^{10}C_{i}u^{i}v^{10-i}.
$$
We are interested in the rational solutions of the system
$$
\cal{S}':\;C_{i}=0, \quad i=0, 1,\ldots, 10,
$$
and, without loss of generality, we can assume that $A_{4}=1$.

Let us observe that the variable $A_{0}$ appears only in one polynomial, i.e., in $C_{5}$, and with the exponent 1. That means that it is enough to find solutions of the system $C_{i}=0$ for $i\in\{0, \ldots, 10\}\setminus\{5\}$, and then compute the value of $A_{0}$ from $C_{5}=0$. We are thus working with the polynomials lying in the ring $\Q[A_{1}, A_{2}, A_{3}, a, b, c, d, e]$.

Let $I=<C_{0},\ldots,C_{4}, C_{6},\ldots, C_{10}>$ denote the ideal generated by the polynomials $C_{i}$ for $i\in\{0, 1, \ldots, 10\}\setminus\{5\}$. One can compute the Gr\"{o}bner basis $\op{Gb}(I)$ of the ideal $I$. It contains 32 polynomials. The computation of the full basis took around 13 minutes. However, for our purposes it is better to work with the basis containing polynomials in variables $A_{3}, e$ only.  We have that
$$
\op{Gb}(I)\cap \Z[A_{3},e]=\{F_{1}, F_{2}, F_{3}\},
$$
where
\begin{equation*}
F_{1}=\left(A_3+1\right) (e-1)^6 \left(e^2+e+1\right) \left(3 e^2+2 e+3\right).
\end{equation*}
Because $A_{3}, e$ need to be rational we have that $(A_{3}+1)(e-1)=0$. First, we consider the case $A_{3}+1=0$. Adding the polynomial $A_{3}+1$ to the ideal $I$ we get the ideal $I'$. The Gr\"{o}bner basis $\op{Gb}(I')$ of the ideal $I'$ takes the form:
\begin{align*}
\op{Gb}(I')=\{ & c-e, (b-d)(e^4-1) (b-d), (b-d)(e^2-1)(d^2+2e), b^2-d^2,\\
               &(b-d)(a^2-d^2-2 e+2),A_3+1, (b-d)(e^2-1)(a d-2A_2),\\
               &(b-d)(ad+A_2 d^2+2eA_2-2 A_2), 2 a A_2-b+d, A_1+A_2\}.
\end{align*}
It is clear that the necessary condition for the existence of rational solutions of the system $\cal{S}'$ is the existence of rational solutions of the system, say $\cal{S}''$, defined by the vanishing of each entry from $\op{Gb}(I')$. Simple computation with the Mathematica procedure {\tt Solve} reveals that $\cal{S}''$ has only one, up to sign changes and permutations, solution which may be useful. This solution has the following form
\begin{equation*}
\begin{array}{llll}
  a=-\sqrt{d^2-4}, & b=-d,                        & c=-1,        & e=-1, \\
  A_{0}=64d(d^2-2),   & A_{1}=\frac{d}{\sqrt{d^2-4}},& A_{2}=-A_{1},& A_{3}=-1.
\end{array}
\end{equation*}
There are also other solutions but they are defined over nontrivial algebraic extensions of $\Q$ or lead to non interesting equalities $s_{1}(u, v)=s_{2}(u, v), s_{3}(u, v)=s_{4}(u, v)$. To make the expression $a=-\sqrt{d^2-4}$ rational one can use the substitution $d=(-1 - t^2)/t$, where $t$ is a rational parameter. Thus, after necessary simplifications we obtain $A(t)=(A_{0}(t), A_{1}(t), A_{2}(t), A_{3}(t), A_{4}(t))$, where
$$
A_{0}(t)=64t^2(1-t^8), \quad A_{1}(t)=-A_{2}(t)=t^2+1,\quad A_{3}(t)=-A_{4}(t)=t^2-1
$$
and note that on the quintic threefold $\cal{V}_{A(t)}$ there is a rational curve defined by the quadratic forms
\begin{equation}\label{svalues}
\begin{cases}
\begin{array}{lll}
  X_{0} & = & s_{0}(u,v)=uv,\\
  X_{1} & = & s_{1}(u,v)=tu^2+(t^2-1)uv+tv^2,\\
  X_{2} & = & s_{2}(u,v)=tu^2-(t^2-1)uv + tv^2,\\
  X_{3} & = & s_{3}(u,v)=tu^2-(t^2+1)uv-tv^2,\\
  X_{4} & = & s_{4}(u,v)=tu^2+(t^2+1)uv-tv^2.
\end{array}
\end{cases}
\end{equation}


We are left with the case $e-1=0$. We apply exactly the same procedure as in the case $A_{3}+1=0$ and find that there are no solutions which lead to nontrivial identities. Summing up, we thus proved

\begin{thm}\label{family}
Let $t\in\Z\setminus\{-1, 0, 1\}$ and put $A(t)=(A_{0}(t), \ldots, A_{4}(t))$. The variety $\cal{V}_{A(t)}$ contains infinitely many non-trivial integer points.
\end{thm}

\begin{rem}
{\rm In this context one can recall that there are quadratic parametrizations living on the diagonal surface $X_{0}^k+X_{1}^k=X_{2}^k+X_{3}^k$, for $k=5, 6, 7$. However, these parametrizations are defined over non-trivial algebraic extension of $\Q$ (see the discussion in \cite{Elk1} and the work of Reznick concerning linear relations between powers of binary quadratic forms \cite{R}).

}
\end{rem}
We define the following counting function
$$
\cal{P}(N)=\#\{A=(A_{0},\ldots, A_{4}):\;\op{max}\{|A_{i}|\}\leq N, \gcd(A)=1,\;\mbox{and}\;\cal{V}_{A}(\Z)\;\mbox{is infinite}\}.
$$
Our findings allow us to deduce the following

\begin{thm}\label{count}
We have $\cal{P}(N)\gg N^{1/10}$.
\end{thm}
\begin{proof}
For a given $t\in\Z\setminus\{-1, 0, 1\}$ we take $A(t)=(A_{0}(t), \ldots, A_{4}(t))$ and observe that $\gcd(A_{0}(t),\ldots, A_{4}(t))=1$ for even values of $t$. If $t$ is odd, we have that $\gcd(A_{0}(t),\ldots, A_{4}(t))=2$. We clearly have $\op{max}\{|A_{i}(t)|:\;i=0, 1, 2, 3, 4\}=64t^2(t^8-1)$. Thus the number of $A$'s such that $\op{max}\{|A_{i}|:\;i=0, 1, 2, 3, 4\}\leq N$ and $\cal{V}_{A}$ has infinitely many non-trivial integer points is not smaller than the number of solutions of the inequality $64t^2(t^8-1)\leq N$, i.e., we have at least $t\gg N^{1/10}$ positive integer values satisfying required properties. We thus have $\cal{P}(N)\gg N^{1/10}$.
\end{proof}

\begin{rem}
{\rm Let us observe that the solutions of $F_{t}(X_{0},\ldots, X_{4})=0$ given by $X_{i}=s_{i}(u, v), i=0,\ldots, 4$, we already constructed, satisfy two additional equations, i.e.,
$$
L(X_{0},\ldots, X_{4})=Q(X_{0},\ldots, X_{4})=0,
$$
where
$$
L=4X_0+X_1-X_{2}+X_3-X_{4},\quad Q=X_1^2+X_2^2-X_3^2-X_4^2.
$$
It is interesting to note that the forms $L, Q$ do not depend on $t$. Thus, for given $A\in\Z^{5}$, it is natural to define the variety
$$
\cal{C}_{A}:\;F_{A}(X_{0},\ldots, X_{4})=L(X_{0},\ldots, X_{4})=Q(X_{0},\ldots, X_{4})=0.
$$
One can check that for a generic choice of the vector $A\in\Z^{5}$, the variety $\cal{C}_{A}$ is a curve of genus $>1$. However, as we have seen the genus of $\cal{C}_{A}$ can drop to 0. From Faltings theorem (see \cite{Fal}) we know that the necessary condition for $\cal{C}_{A}$ to have infinitely many rational points is when the genus is $\leq 1$.
We thus formulate the following

\begin{prob}\label{genus01}
Characterize all 5-tuples of integers $A=(A_{0},\ldots, A_{4})$ such that the curve $\cal{C}_{A}$ has genus $\leq 1$.
\end{prob}
}
\end{rem}

\section{An application}\label{sec4}

In this section we present a curious application of our construction of solutions of the equation defining $\cal{V}_{A(t)}$. More precisely, we construct infinitely many diagonal quintic forms $G_{t}, t\in\Z\setminus\{-1, 0, 1\}$, in four variables, and such that the equation $G_{t}(X_{0}, X_{1}, X_{2}, X_{3})=R(t)$ has infinitely many integer solutions for suitable chosen integer $R(t)$.

Before we go on, let us recall the following well know property of Pell type equation $V^2-dU^2=a$, where $a\in\Z\setminus\{0\}$ and $d\in\N$ is not a square. If $(u, v)$, where $uv\neq 0$, is an integer solution of this equation, and $(X_{0}, Y_{0})$ is the fundamental solution of the Pell equation $Y^2-dX^2=1$, then for each $n\in\N_{+}$, the pair $(U_{n}, V_{n})$ defined by the equality
$$
V_{n}+U_{n}\sqrt{d}=(v+u\sqrt{d})(Y_{0}+X_{0}\sqrt{d})^{n}
$$
is a solution of the equation $V^2-dU^2=a$. In particular, the existence of one solution $(u, v), uv\neq 0$, of the equation $V^2-dU^2=a$ implies the existence of infinitely many solutions.

\bigskip
\begin{lem}\label{lem1}
The equation
\begin{equation}\label{impeq}
s_{3}(u,v)=s_{3}(1,1)=-t^2-1,
\end{equation}
where $s_{3}(u, v)$ is given in (\ref{svalues}), has infinitely many solutions $u, v$ in the ring $\Z\left[\frac{1}{2},t\right]$.
\end{lem}
\begin{proof}
 Let us note the equation (\ref{impeq}) can be written in the following form
\begin{equation}\label{impeq2}
V^2-(t^4+6t^2+1)U^2=4t(t^2+1),
\end{equation}
where $V=2tv+(1+t^2)u, U=u$. Let us write $P(t)=t^4+6t^2+1$. Thus, we consider Pell type equation and look for solutions $U, V$ satisfying congruence condition $V\equiv (1+t^2)U\pmod*{2t}$.

Let us note that (\ref{impeq2}) has a non trivial solution $(V, U)=((t+1)^2, 1)$. Moreover, the  Pell equation
$$
Y^2-P(t)X^2=1
$$
has the fundamental solution
$$
Y_{0}=Y_{0}(t)=\frac{1}{4}(t^2+1)(t^2+5),\quad X_{0}=X_{0}(t)=\frac{1}{4}(t^2+3).
$$
As a consequence, we get that for each $n\in\N$ the pair $(V_{n}, U_{n})$ defined by
$$
V_{n}+U_{n}\sqrt{P(t)}=\left((t+1)^2+\sqrt{P(t)}\right)^{n}\left(\frac{1}{4}(t^2+1)(t^2+5)+\frac{1}{4}(t^2+3)\sqrt{P(t)}\right),
$$
is a solution of (\ref{impeq2}). Next, to obtain solutions of (\ref{impeq}) we need to know that $(V_{n}\equiv (1+t^2)U_{n})/2t\in\Z\left[\frac{1}{2},t\right]$. However, a simple induction on $n$ confirms that this is the case. Finally, for each $n\in\N$ the pair $(v_{n}, u_{n})$, where
\begin{equation}\label{valuesuv}
u_{n}=U_{n}, \quad v_{n}=\frac{1}{2t}(V_{n}-(1+t^2)U_{n}),
\end{equation}
is a solution of (\ref{impeq}). To see an example we take $n=1$ and get
$$
u_{1}=\frac{1}{2}(t^4+t^3+5 t^2+3 t+4), \quad v_{1}=\frac{1}{2}(t^3+t^2+3 t+1).
$$
For $n=2$ we get
\begin{align*}
u_{2}&=\frac{1}{4} \left(t^8+t^7+11 t^6+9 t^5+39 t^4+23 t^3+49 t^2+15 t+16\right), \\
v_{2}&=\frac{1}{4} \left(t^7+t^6+9 t^5+7 t^4+23 t^3+11t^2+15 t+1\right).
\end{align*}
\end{proof}

\begin{rem}
{\rm Using standard methods it is not difficult to prove that the following continued fraction expansion
$$
\sqrt{(2 t+1)^4+6 (2 t+1)^2+1}=[4t^2+4t+3; \overline{1, t^2+t-1, 1, 2(4t^2+4t+3)}]
$$
holds. Thus, if we replace $t$ by $2t+1$ in (\ref{impeq}), then the resulting equation has infinitely many solutions in the ring $\Z[t]$.
}
\end{rem}

We are ready to prove the following.

\begin{thm}\label{app1}
Let $G_{t}(X_{0}, X_{1}, X_{2}, X_{3})=(t^2+1)(X_0^5-X_1^5)+(1-t^2)X_2^5+64t^2(1-t^8)X_3^5$. For each $t\equiv 1\pmod*{2}$, the Diophantine equation
$$
G_{t}(X_{0}, X_{1}, X_{2}, X_{3})=(t^2-1)(t^2+1)^5
$$
has infinitely many solutions in integers.
\end{thm}
\begin{proof}
To get the proof it is enough to use the polynomials $u_{n}, v_{n}$ given by (\ref{valuesuv}). Indeed, we know that $s_{3}(u_{n}, v_{n})=-t^2-1$. Thus, for each $n\in\N_{+}$ we can take
$$
X_{0, n}=s_{1}(u_{n}, v_{n}), X_{1, n}=s_{2}(u_{n}, v_{n}), X_{2, n}=s_{4}(u_{n}, v_{n}), X_{3, n}=s_{0}(u_{n}, v_{n})
$$
and observe that $G_{t}(X_{0, n}, X_{1, n}, X_{2, n}, X_{3, n})=(t^2-1)(t^2+1)^5$.
\end{proof}

For $B=(B_{0}, B_{1}, B_{2}, B_{3})\in\Z^{4}$ let us put
$$
G_{B}(X_{0}, X_{1}, X_{2}, X_{4})=B_{0}X_{0}^5+B_{1}X_{1}^5+B_{2}X_{2}^5+B_{3}X_{3}^5
$$
and define the counting function
$$
C_{b,B}(N)=\#\{(X_{0}, X_{1}, X_{2}, X_{3})\in\N:\; |X_{i}|\leq N \;\mbox{and}\; G_{B}(X_{0}, X_{1}, X_{2}, X_{3})=b\}.
$$
From Theorem \ref{app1} we immediately get the following.
\begin{cor}
There are infinitely many $B\in\Z^{4}$ satisfying $\gcd(B)=1$ and such that there is an integer $b$ depending on $B$, such that
$$
C_{b, B}(N)\gg \log N.
$$
\end{cor}
\begin{proof}
For $t\in\Z\setminus\{-1, 0, 1\}$ take $B=(t^2+1, -t^2-1, 1-t^2, 64t^2(1-t^8))$ and $b=(t^2-1)(t^2+1)^5$.
\end{proof}

\begin{exam}
{\rm
We take $t=3$. From Theorem \ref{app1}, we know that the equation
\begin{equation}\label{exameq}
5X_0^5-5X_1^5-4X_2^5-1889280X_3^5=2^7 5^5
\end{equation}
has infinitely many solutions in positive integers. Using the formulas for $X_{i, n}, i=0, 1, 2, 3$, from the proof of Theorem \ref{app1}, one can compute small solutions of our equation. The solutions corresponding to $n=2, 3$ are
$$
(37526, 6982, 38170, 1909), \quad (183773534, 34226638, 186933610, 9346681).
$$

Note that the obtained family of integer solutions does not cover all integer solutions of (\ref{exameq}). Indeed, due to reducibility of the quintic form $5(X_{0}^5-X_{1}^5)$, one can find all integer solutions of (\ref{exameq}) satisfying the condition $\op{max}\{|X_{2}|, |X_{3}|\}\leq 10^3.$ In this range we have the solutions
\begin{align*}
&(-166, -38, -170, -8), (38, 166, -170, -8), (2, -14, 10, 1),\; (14, -2, 10, 1),\\
&(-10,-90, 90, 5),\;(90, 10, 90, 5), (-94, -542, 550, 28),\; (542, 94, 550, 28).
\end{align*}

}
\end{exam}

\begin{rem}
{\rm According to our best knowledge, in the above result we have the first explicit example of a diagonal quintic form $F$ in four variables and such that there is an integer $b$ for which the Diophantine equation $F(X_{0}, X_{1}, X_{2}, X_{3})=b$ has infinitely many integer solutions satisfying the condition $X_{0}X_{1}X_{2}X_{3}\neq 0$. }
\end{rem}

\section*{Appendix}

Coefficients $C_{i}$ (up to constant factor) in the expansion of the polynomial
$$
\sum_{i=0}^{4}A_{i}s_{i}(u,v)^{5}=\sum_{j=0}^{10}C_{j}u^{10-j}v^{j}
$$
are given by:
\begin{align*}
C_{0}&=A_1+A_2+A_3+1, \\
C_{1}&=-a A_1+a A_2+A_3 b+d, \\
C_{2}&=(2 a^2+1) A_1+2 a^2 A_2+(2 b^2+c)A_{3}+2d^2+e,\\
C_{3}&=-a(a^2+2)A_1+a(a^2+2)A_2+b(b^2+c)A_3+d^3+2de, \\
C_{4}&=(a^4+6a^2+2)(A_{1}+A_{2})+(b^4+6 b^2 c+2 c^2)A_{3}+d^4+6d^2e+2e^2, \\
C_{5}&=A_{0}+a(a^4+20a^2+30)(A_{2}-A_{1})\\
     &\quad +b(b^4+20 b^2 c+30c^2)A_{3}+d(d^4+20d^2e+30e^2),\\
C_{6}&=(a^4+6 a^2+2)(A_{1}+A_{2})+c(b^4+6 b^2 c+2 c^2)A_{3}+e(d^4+6d^2e+2e^2), \\
C_{7}&=-a(a^2+2)A_{1}+a(a^2+2)A_{2}+bc^2(b^2+2c)A_{3}+de^2(d^2+2e),\\
C_{8}&=(2a^2+1)A_1+(2a^2+1)A_2+c^3(2b^2+c)A_3+2 d^2 e^3+e^4, \\
C_{9}&=-a A_1+a A_2+A_3 b c^4+d e^4, \\
C_{10}&=c^{5}A_3+A_1+A_2+e^5. \\
\end{align*}
The computation of the Gr\"{o}bner basis $\op{Gb}(I)$ of the ideal $I=<C_{0},\ldots, C_{4}, C_{6}, \ldots, C_{10}>$ was performed with the procedure
\begin{center}
{\tt GroebnerBasis$[I, \{A_{1}, A_{2}, A_{3}, a, b, c, d, e\}]$}.
\end{center}
The computation took around 17 minutes.

The computation of the intersection $\op{Gb}(I)\cap \Q[A_{3}, e]$ was performed with the procedure
\begin{center}
{\tt GroebnerBasis$[I, \{A_{1}, A_{2}, A_{3}, a, b, c, d, e\},\{A_{1}, A_{2}, a, b, c, d\}]$}.
\end{center}
The computation took less than 7 minutes.

The computation of the Gr\"{o}bner basis $\op{Gb}(I')$ of the ideal $I'=<A_{3}+1, I>$ was performed with the procedure
\begin{center}
{\tt GroebnerBasis$[I', \{A_{1}, A_{2}, A_{3}, a, b, c, d, e\}]$}.
\end{center}
The computation took less than a second. To solve the resulting equations from the set $\op{Gb}(I')=\{f_{1},\ldots, f_{10}\}$, we used the procedure
\begin{center}
{\tt Solve$[\{f_{1},\ldots, f_{10}\}=\{0,\ldots, 0\},\{A_{1}, A_{2}, A_{3}, a, b, c, d, e\}]$}.
\end{center}

\bigskip

\noindent {\bf Acknowledgements.} The author is grateful to two anonymous referees for a careful reading of the paper and many suggestions which improved the presentation.

\bigskip

\noindent Maciej Ulas, Jagiellonian University, Institute of Mathematics,
{\L}ojasiewicza 6, 30-348 Krak\'ow, Poland; email:\;{\tt maciej.ulas@uj.edu.pl}

 \end{document}